\documentclass[12pt]{amsart}

\usepackage{mathtools} 
\usepackage[top=100pt,bottom=60pt,left=96pt,right=92pt]{geometry}
\usepackage[mathscr]{eucal}
\usepackage{hyperref} 
\usepackage{amsmath}
\usepackage{amsthm}
\usepackage{amssymb}
\usepackage{amsfonts} 
\usepackage{mathrsfs}
\usepackage{amscd}
\usepackage{enumitem} 

\usepackage[pdftex]{color}
\usepackage{xcolor}
\usepackage{layout}
\usepackage{stmaryrd} 
\usepackage[pdftex]{graphicx}
\usepackage{txfonts} 
\usepackage{verbatim}
\usepackage{textcomp}

\numberwithin{equation}{section}

\newtheorem{theorem}{Theorem}[section]
\newtheorem{proposition}[theorem]{Proposition}
\newtheorem{corollary}[theorem]{Corollary}

\newtheorem{definition}[theorem]{Definition}

\theoremstyle{plain}

\newcommand{\vungoc}{V\~u Ng\d{o}c }

\newcommand{\R}{\mathbb{R}}
\newcommand{\mbS}{\mathbb{S}} 
\newcommand{\T}{\mathbb{T}} 
\newcommand{\Z}{\mathbb{Z}}

\def\epsilon{\varepsilon}
\def\phi{\varphi}

\newcommand{\al}{\alpha}

\newcommand{\om}{\omega}

\newcommand{\mcF}{\mathcal F}

\newcommand{\mcM}{\mathcal M}

\DeclareMathOperator{\Id}{Id}

\DeclareMathOperator{\Span}{Span}

\allowdisplaybreaks

\newcommand{\refperiodCost}{Definition \ref{periodCost}}

\newcommand{\refzeroCosts}{Proposition \ref{zeroCosts}}

\newcommand{\refgeneralToric}{Corollary \ref{generalToric}}

\begin{document}

\title[Toric systems via transport costs]{Characterization of toric systems via transport costs}

\author{Sonja Hohloch}

\date{\today}

\begin{abstract}
 We characterize completely integrable Hamiltonian systems inducing an effective Hamiltonian torus action as systems with zero transport costs w.r.t.\ the time-$T$ map where $T\in \R^n$ is the period of the acting $n$-torus.
\end{abstract}

\maketitle


\section{Introduction}

Integrable Hamiltonian systems play an important role in physics and mathematics since they describe systems with symmetries. During the last 3-4 decades, there have been several breakthroughs in terms of achieving local or global, topological or symplectic classifications of certain types of integrable systems. 

Being far from exhaustive, let us just mention the symplectic classifications of interest for this short note: There is Delzant's \cite{delzant} symplectic classification of those integrable systems that induce an effective Hamiltonian torus action --- the classifying invariant hereby is the image of the momentum map. The singular points of toric or toric type systems only admit elliptic and regular components, but no focus-focus or hyperbolic components --- which may occur at nondegenerate singularities of integrable systems (see the local normal form by Eliasson \cite{eliasson}, Miranda $\&$ Zung \cite{miranda-zung}, and others). In addition to the singularities appearing in toric systems, focus-focus singularities may occur in addition in so-called semitoric systems who have been classified by Pelayo $\&$ \vungoc\ \cite{pelayo-vungoc2009, pelayo-vungoc2011} on four dimensional symplectic manifolds.

In the present note, we give a characterization of toric systems by means of `transport costs' which measure `how toric or nontoric' a system is: toric systems on $2n$-dimensional manifolds are precisely those systems having zero transport costs w.r.t.\ the time-$T$ map where $T \in \R^n$ is the period of the acting $n$-torus.
To be more precise, we introduce the notion of {\em periodicity costs} (see \refperiodCost) and characterize toric systems as having zero periodicity costs (see \refzeroCosts\ and \refgeneralToric).

This approach is inspired by the techniques developed around the transport problems of Monge \cite{monge} and Kantorovich \cite{kantorovich42, kantorovich48} where, for given cost functions, optimal transport functions resp.\ transport measures are looked for. Transport problems have been vividly studied during the last 3-4 decades, see e.g.\ the monographs by Rachev $\&$ R\"uschendorf \cite{rachev-rueschendorf} and Villani \cite{villani} for an overview.

Since transport problems `live naturally' within the calculus of variation, but integrable systems, due to their rigidity, usually do not at all `fit well together' with variational methods, it is quite astonishing that `periodicity costs' in the sense of transport costs provide a meaningful notion for integrable systems, even singling out the subclass of toric systems.

In future projects, we hope to find an answer to e.g.\ the following questions:
\begin{enumerate}[label={\arabic*)}]
 \item 
 Can one use the notion of (localized?) periodicity costs to investigate the behaviour around focus-focus points in semitoric systems?
 \item
 Do periodicity costs relate to the symplectic invariants classifying semitoric systems, in particular to the Taylor series invariant?
 \item
 Are there more methods and ideas from transport theory that have an application within integrable systems?
\end{enumerate}


\subsection*{Acknowledgements}
The author wishes to thank Marine Fontaine for helpful remarks. The work on this paper was partially supported by the FWO-EoS project G0H4518N and the UA-BOF project with Antigoon-ID 31722.


\section{Notions and conventions}

\noindent
Within this section, let $(M, \om)$ be a $2n$-dimensional compact symplectic manifold. 


\subsection{Basic definitions and conventions}

The {\em Hamiltonian vector field} $X^f$ of a smooth function $f: M \to \R$ is defined by $\om(X^f, \cdot) = -df$. The flow of the associated {\em Hamiltonian equation} $z' = X^f(z)$ is called {\em Hamiltonian flow of $f$} and denoted by $\phi^f$. It is a smooth map $\phi^f: \R \times M \to M$ where we usually write $\phi^f(t, p)=:\phi^f_t(p)$ with $t \in \R$ and $p \in M$. Hence, for all $t \in \R$, we get a diffeomorphism $\phi^f_t: M \to M$ that is in fact symplectic and often referred to as {\em time-$t$ map}.

Given two smooth functions $f, g: M \to \R$, their {\em Poisson bracket induced by $\om$} is defined as $\{f, g\}:= \om(X^f, X^g) = -df(X^g)=dg(X^f)$. The smooth functions $f, g: M \to \R$ are said to {\em Poisson commute} if $\{f, g\}=0$. 

In our sign convention, the Lie bracket and the Poisson bracket are related via $[X^f, X^g]=X^{-\{f, g\}}$. Poisson commutativity of $f$ and $g$ implies $[X^f, X^g]=0$ and thus commutativity of their Hamiltonian flows, i.e., $\phi^f_s \circ \phi^g_t = \phi^g_t \circ \phi^f_s$ for all $s, t \in \R$.

For more details, proofs, and further reading, we refer the interested reader e.g.\ to McDuff $\&$ Salamon \cite{mcduff-salamon}.


\subsection{Integrable Hamiltonian systems}

Recall that $\dim M = 2n$.
A smooth function $h:=(h_1, \dots, h_n) : M \to \R^n$ is said to be a {\em (momentum map of a) $2n$-dimensional completely integrable Hamiltonian system} if $X^{h_1}$, \dots, $X^{h_n}$ are almost everywhere linearly independent and if $\{h_i, h_j\}=0$ for all $1 \leq i, j \leq n$.

We briefly write $(M, \om, h)$ for a completely integrable system on $(M, \om)$ with momentum map $h$. We speak of a {\em compact} completely integrable system if we want to emphasize that the underlying symplectic manifold $(M, \om)$ is compact.

We call $p \in M$ a {\em regular} point of a completely integrable systems $(M, \om, h)$ if $\dim \bigl(\Span\{ X^{h_1}(p), \dots, X^{h_n}(p)\} \bigr) = n$ and {\em singular} otherwise. We denote the set of regular points by $M^{reg}$ and the set of singular points by $M^{sing}$. 
The property {\em almost everywhere} in the definition of an integrable system is understood w.r.t.\ to the measure $\mu_\om$ induced by the $n$-fold wedge product $\om^n$ of $\om$ (which in turn is continuous w.r.t.\ the Lebesgue measure and vice versa). Thus $\mu_\om(M) = \mu_\om(M^{reg})$ and $\mu_\om(M^{sing})=0$.

Let $t:=(t_1, \dots, t_n) \in \R^n$ and let $\al$ be a permutation of the set $\{1, \dots, n\}$. Commutativity of the flows of the component functions $h_1, \dots, h_n$ implies
$$  \phi_{t_1}^{h_1} \circ \cdots \circ \phi_{t_n}^{h_n} =  \phi_{t_{\al(1)}}^{h_{\al(1)}} \circ \cdots \circ \phi_{t_{\al(n)}}^{h_{\al(n)}}.$$
Thus a completely integrable system induces a welldefined (Hamiltonian) action of the abelian group $(\R^n, +)$ on $M$ via 
$$\R^n \times M \to M, \quad (t.p):= \phi^h_t(p):= \phi_{t_1}^{h_1} \circ \cdots \circ \phi_{t_n}^{h_n} .$$
For details, proofs, and further reading, we refer e.g.\ to Bolsinov $\&$ Fomenko \cite{bolsinov-fomenko}.


\subsection{Toric systems}

Let $\T^1 = \mbS^1 = \R \slash 2\pi \Z$ and $\T^n = (\mbS^1)^n$ and recall that a group action is said to be {\em effective} or {\em faithful} if the identity element is the only one acting trivially. 

A $2n$-dimensional completely integrable system $(M, \om, h=(h_1, \dots, h_n))$ is {\em toric} if the action induced by the flow $\phi^h: \R \times M \to M$ is an effective (Hamiltonian) $\T^n$-action, i.e., $\phi^{h_k}_{2 \pi}= \Id$ for all $k \in \{1, \dots, n\}$ and $2\pi$ is the `minimal common period' of the $n$ Hamiltonian circle actions induced by $h_1, \dots, h_n: M \to \R$.

For more details, proofs, and further reading, we refer e.g.\ to Cannas da Silva \cite{cannas}.


\section{Optimal transport}

Let $M$ be a compact manifold and $\mu_-$ and $\mu_+$ two (positive) measures with same total mass $\mu_-(M) = \mu_+(M)< \infty$. Let $c: M \times M \to \R^{\geq 0}$ a `sufficiently regular' function, usually referred to as {\em cost function}.


\subsection{The Monge transport problem}

Consider the {\em space of transport maps} 
$$\mcF(\mu_-, \mu_+):=\{ f: M \to M \mid f \mbox{ measurable}, f(\mu_-) =\mu_+ \}$$ 
where $f(\mu_-)$ denotes the image or push forward measure of $\mu_-$ under $f$.
Formulated in modern language, the French mathematician Monge \cite{monge} asked in 1781 if there is $f \in \mcF(\mu_-,\mu_+)$ minimizing 
$$
\int_Mc(x, f(x)) \ d\mu_-
$$
over $\mcF(\mu_-, \mu_+)$. This is a nonlinear optimization problem and usually referred to as {\em Monge transport problem}.

Note that, if $\mu_-$ contains point measures, there does not necessarily exist a transport {\em map} since the requirement $f(\mu_-) =\mu_+$ may force the mass of {\em one} point to be distributed over {\em several distinct} points.

In 1979, Sudakov \cite{sudakov} proposed a proof of Monge's problem in $\R^n$ with the Euclidean distance as cost function. Unfortunately the proof turned out to have a gap (cf.\ Ambrosio \cite[p.\ 137]{ambrosio1}, \cite[Chapter 6]{ambrosio2}) that can only be mended under stronger assumptions.

For more details and references, we refer the reader e.g.\ to the monographs by Rachev $\&$ R\"uschendorf \cite{rachev-rueschendorf} and Villani \cite{villani}.


\subsection{The Kantorovich transport problem}

Let $p_-, p_+: M \times M \to M$ be the projections on the first and second factor respectively and consider the {\em space of transport measures} 
$$
\mcM(\mu_-, \mu_+):= \{\mu \mbox{ measure on } M \times M \mid  p_-(\mu) = \mu_-, \ p_+(\mu) = \mu_+\}.
$$ 
Then the search for a measure $\mu \in \mcM(\mu_-, \mu_+)$ minimizing
$$
\int_{M \times M} c(x, y) \ d\mu
$$
over $\mcM(\mu_-, \mu_+)$ is referred to as solving the {\em Kantorovich transport problem}. It is named after the Russian mathematician Kantorovich \cite{kantorovich42, kantorovich48} and it enjoys much more `analysis friendly' properties than Monge's problem, for details see e.g.\ Rachev $\&$ R\"uschendorf \cite{rachev-rueschendorf} and Villani \cite{villani}. 

Both transport problems are related as follows:
If $f\in \mcF(\mu_-, \mu_+)$ is a solution of Monge's problem, then $(\Id \times f)(\mu_-) \in \mcM(\mu_-, \mu_+) $ such that
\begin{align*}
\inf_{ f \in \mcF(\mu_-, \mu_+)} \int_Mc(x, f(x)) \ d\mu_- & = \inf_{ f \in \mcF(\mu_-, \mu_+)} \int_{M \times M} c(x, y) \ d \bigl( (\Id \times f)(\mu_-) \bigr)  \\
& \geq \inf_{\mu \in \mcM(\mu_-, \mu_+)} \int_{M \times M} c(x, y) \ d\mu.
\end{align*}
i.e., a solution of Kantorovich's problem gives a lower bound for Monge's problem. Under certain convexity and growth assumptions on the cost function, Gangbo $\&$ McCann \cite{gangbo-mccann} stated an explicit formula for the optimal transport map and showed that optimal transport measures `lie in the graph' of the optimal transport map.


\section{Characterization of toric systems by transport costs}

\noindent
Throughout this section, let $(M, \om)$ be a compact $2n$-dimensional symplectic manifold. 

In what follows, we consider the following modified transport problem: Given a cost function and a certain type of transport maps, how does the `minimal' transport map {\em within this type of transport maps} look like?


\subsection{The cost functional}

Let $(M, \om, h)$ be a completely integrable system. Let $t=(t_1, \dots, t_n) \in \R^n$ and recall that the time-$t$ map $\phi_t^h: M \to M$ is symplectic. The $n$-fold wedge product $\om^n$ of $\om$ is a volume form on $M$ that is invariant under the time-$t$ map, i.e., $\left(\phi^h_t \right)^* \om^n = \om^n$. When we consider $\om^n$ as measure on $M$ we write $\mu_\om$. The image measure under the time-$t$ map satisfies $\phi^h_t(\mu_\om) = \mu_\om$.

Let $c: M \times M \to \R^{\geq 0} $ be a continuous cost function and let $U \subseteq M$ be open. Define the parameter depending integral
\begin{equation*}
C_t^h(U, c):=\int_U c(x, \phi_t^h(x)) \ d\mu_\om.
\end{equation*}
The map 
$$c \circ (\Id \times\ \phi_t^h ) \ : \ \R \to \R^{\geq 0}, \qquad t \mapsto c(x, \phi_t^h(x))$$ 
is continuous for all $x \in M$ and so is $\R \to \R^{\geq 0}$, $t \mapsto C_t^h(U, c)$ for all open $U \subseteq M$.


\subsection{Characterization of toric systems}

We begin with

\begin{definition}
 Let $M$ be a manifold. A function $c: M \times M \to \R^{\geq 0}$ is {\em metric-like} if 
\begin{enumerate}[label={\arabic*)}]
 \item 
 $c(x, y) = 0$ for $x, y \in M$ if and only if $x=y$.
 \item
 $c(x, y) = c(y, x)$ \ for all $x, y \in M$.
\end{enumerate}
\end{definition}

For instance, the Euclidean distance is a continuous metric-like cost function. Its square is a smooth metric-like cost function.

The functional $ C_t^h(U, c)$ can be used to measure `how periodic' a given completely integrable system is: 

\begin{definition}
\label{periodCost}
 Let $(M, \om, h)$ be a $2n$-dimensional compact completely integrable system and $c: M \times M \to \R^{\geq 0}$ a continuous metric-like cost function and $T \in \R^n$. We call $C_T^h(M, c) \in [0, \infty[$ the {\em $T$-periodicity costs} of $(M, \om, h)$ w.r.t.\ the cost function $c$.
\end{definition}

Now we will see what these notions mean for toric systems.

\begin{proposition}
\label{zeroCosts}
 Let $(M, \om, h)$ be a $2n$-dimensional compact completely integrable system and $c: M \times M \to \R^{\geq 0}$ a continuous metric-like cost function.
 Then $(M, \om, h)$ is toric if and only if $C_{(2\pi, \dots, 2 \pi)}^h(M, c)=0$ and $C_{s}^h(M, c) > 0$ for all $s \in \R^n \setminus\{(2\pi k, \dots, 2 \pi k) \mid k \in \Z \} $.
\end{proposition}

\begin{proof}
 $(M, \om, H)$ being a toric system means that the flow $\phi^h$ is $(2\pi, \dots, 2\pi)$-periodic and the action is effective, i.e., $(2\pi, \dots, 2\pi)$ is minimal in the sense that $\phi^h_s \neq \Id$ for all $s \in [0, 2\pi]^n \setminus\{(0, \dots, 0), (2\pi, \dots, 2 \pi )\} $.

 {`$\Rightarrow$':} 
 Since $\phi^h_{(2 \pi, \dots, 2 \pi)}(x) = x$ for all $x \in M$ we obtain $c(x, \phi^h_{(2 \pi, \dots, 2 \pi)}(x))=0$ for all $x \in M$ and thus $C_{(2 \pi, \dots, 2 \pi)}^h(M, c) =0$. Now let $s \in [0, 2\pi]^n \setminus\{(0, \dots, 0), (2\pi, \dots, 2 \pi )\} $. Since $\phi^h_s \neq \Id$ there exists $y \in M$ with $\phi_s^h(y) \neq y$. Continuity of $\phi^h_s$ implies the existence of an open neighbourhood $U \subseteq M$ of $y$ with $\phi^h_s(z) \neq z$ for all $z \in U$. Since $c$ is metric-like, $c(z, \phi^h_s(z)) >0$ for all $z \in U$ and thus $0 < C_s^h(U, c) \leq C_s^h(M, c)$. The $(2 \pi, \dots, 2 \pi)$-periodicity of $\phi^h$ implies that this is true for all $s \in \R^n \setminus\{(2\pi k, \dots, 2 \pi k) \mid k \in \Z \} $.

 {`$\Leftarrow$':} Assume that $\phi^h_{(2 \pi, \dots, 2 \pi)} \neq \Id$. Then there exists $x \in M$ with $\phi_{(2 \pi, \dots, 2 \pi)} ^h(x) \neq x$. Metric-likeness implies $c(x, \phi_{(2 \pi, \dots, 2 \pi)} ^h(x))>0$. Hence, since $c$ and $\phi^h$ are continuous, there exists an open neighbourhood $U$ of $x$ with $\phi_{(2 \pi, \dots, 2 \pi)}^h(y) \neq y$ for all $y \in U$ and thus $c(y, \phi_{(2 \pi, \dots, 2 \pi)}^h(y))>0$ for all $y \in U$. Therefore $0< C_{(2 \pi, \dots, 2 \pi)}^h(U, c) \leq C_{(2 \pi, \dots, 2 \pi)}^h(M, c) $ \ $\lightning$.
 
 Since $\phi^h_{(2 \pi, \dots, 2 \pi)} = \Id$, it suffices to show $\phi^h_s \neq \Id$ for all $s \in [0, 2\pi]^n \setminus\{(0, \dots, 0), (2\pi, \dots, 2 \pi )\} $ to prove the claim for all $s \in \R^n \setminus\{(2\pi k, \dots, 2 \pi k) \mid k \in \Z \} $. Since $C_{s}^h(M, c) > 0$ there exists $V \subseteq M$ with $\mu_\om(V)>0$ and $c(z, \phi^h_s(z))> 0$ for all $z \in V$. Because of $\mu_\om(M^{reg})= \mu_\om(M)$ we have $V \cap M^{reg} \neq \emptyset$ and $\mu_\om(V \cap M^{reg})=\mu_\om(V)>0$ and in particular $c(z, \phi^h_s(z))> 0$ for all $z \in V \cap M^{reg}$, i.e., $\phi^h_s(z) \neq z$ for all $z \in V \cap M^{reg}$, i.e., $\phi^h_s \neq \Id$. 
\end{proof}

If we do not work with the normalization $\T^n = (\R \slash 2 \pi \Z)^n$, we obtain more generally:

\begin{corollary}
\label{generalToric}
Completely integrable Hamiltonian systems with periodic flow are characterized by vanishing transport costs w.r.t.\ the time-$T$ map where $T\in \R^n$ is the period of the (effectively) acting $n$-torus.
\end{corollary}


\vspace{10mm}

\begin{tabular}{ll}
 \textit{Postal address:} & Sonja Hohloch\\
 & Department of Mathematics\\
 & University of Antwerp\\
& Middelheimlaan 1\\
& B-2020 Antwerp, Belgium\\
& \\
\textit{E-mail:} & \texttt{sonja.hohloch AT uantwerpen.be}
\end{tabular}

\end{document}